\documentclass[12pt]{amsart}
\usepackage{amssymb}
\usepackage[all]{xy}

\input xy
\xyoption{all}

\newtheorem{lemma}{Lemma}[section]
\newtheorem{corollary}[lemma]{Corollary}
\newtheorem{theorem}[lemma]{Theorem}
\newtheorem{proposition}[lemma]{Proposition}

\theoremstyle{definition}
\newtheorem{remark}[lemma]{Remark}
\newtheorem{definition}[lemma]{Definition}

\newtheorem{example}[lemma]{Example}

\begin{document}
\title{The commutative core of a Leavitt path algebra}

\author{Crist\'{o}bal Gil Canto}
\address{Departamento de \'{A}lgebra, Geometr\'{\i}a y Topolog\'{\i}a\\
Universidad de M\'{a}laga\\
29071, M\'{a}laga, Spain}
\email{cristogilcanto@gmail.com / cgilc@uma.es}

\author{Alireza Nasr-Isfahani}
\address{Department of Mathematics\\
University of Isfahan\\
P.O. Box: 81746-73441, Isfahan, Iran\\ and School of Mathematics, Institute for Research in Fundamental Sciences (IPM), P.O. Box: 19395-5746, Tehran, Iran}
\email{nasr$_{-}$a@sci.ui.ac.ir / nasr@ipm.ir}



\begin{abstract} For any unital commutative ring $R$ and for any graph $E$, we identify the commutative core of the Leavitt path algebra of $E$ with coefficients in $R$, which is a maximal commutative subalgebra of the Leavitt path algebra. Furthermore, we are able to characterize injectivity of representations which gives a generalization of the Cuntz-Krieger uniqueness theorem.
\end{abstract}

\maketitle


\section{Introduction}

For a commutative unital ring $R$, the Leavitt path algebras are a specific type of path $R$-algebras associated to a graph $E$, modulo some relations. Initially, the Leavitt path algebra $L_K(E)$ was introduced for row-finite graphs (countable graphs such that every vertex emits only a finite number of edges) and for an arbitrary field $K$ in \cite{AA1} and \cite{AMP}; later extending the definition to arbitrary graphs in \cite{AA2}. Tomforde in \cite{T} generalizes the construction of Leavitt path algebras by replacing the field $K$ with a commutative unital ring $R$.

Leavitt path algebras are the algebraic version of Cuntz-Krieger graph $C^*$-algebras: a class of algebras intensively investigated by analysts for more than two decades (see \cite{R} for an overview of the subject). On the other hand, they can be considered as natural generalizations of Leavitt algebras $L(1,n)$ of type $(1,n)$, investigated by Leavitt in \cite{L} in order to give examples of algebras not satisfying the IBN property (i.e., whose modules have bases with different cardinality).

Leavitt path algebras of graphs include many well-known algebras such as matrix rings ${\mathbb M}_n(R)$ for $n\in \mathbb{N}$, the Laurent polynomial ring $R[x,x^{-1}]$, or the classical Leavitt algebras $L(1,n)$ for $n\ge 2$.

The algebraic and analytic theories share important similarities, but also present some remarkable differences. The relation between these two classes of graph algebras has been mutually beneficial. This is the case once more for the topic discussed in the current paper: the analytic result was given in \cite{NR} for the $C^*$-algebras $C^*(E)$, and we give here the algebraic analogue for the Leavitt path algebras $L_R(E)$.

The paper is organized as follows. In Section 2 we give all the
background information, together with the definition of $L_R(E)$ and
some basic properties. Also we introduce Cuntz-Krieger $E$-systems,
the analog relations of Leavitt path algebras in a general
$R$-algebra. In Section 3 we give a particular representation of the
Leavitt path algebra (Proposition \ref{Piaprepresentation}) on a
specific $R$-algebra related to the set of all essentially aperiodic
trails in the graph, that is, the set of all infinite aperiodic
paths, as well as those infinite paths that are periodic (those that end
in a cycle without exits) together with all finite paths whose range
is a sink.

In Section 4 we introduce the commutative core $M_R(E)$ of the
Leavitt path algebra $L_R(E)$ and we prove one of the main results of
this paper: Theorem \ref{Mismaximal}. This theorem says that
$M_R(E)$ is a maximal commutative subalgebra inside $L_R(E)$. For
this purpose we previously show Theorem \ref{EM}, where we prove the
existence of an algebraic conditional expectation onto the
subalgebra $M_R(E)$.

Finally Section 5 is dedicated to the uniqueness theorems for
Leavitt path algebras, i.e., those which set conditions on the graph
$E$ or on the map $\Phi$ in order to ensure that a representation
$\Phi:L_R(E) \rightarrow \mathcal{A}$ is injective. Analogously to
\cite[Theorem 3.13]{NR} where it is proved a uniqueness theorem in
the spirit of Szymanski's Uniqueness Theorem for graph
$C^{\ast}$-algebras (see \cite[Theorem 1.2]{S}), Theorem
\ref{UniquenessThm} says in particular that a representation of
$L_R(E)$ is injective if and only if it is injective on the
commutative core $M_R(E)$. This theorem also generalizes
\cite[Theorem 3.7]{AMMS} for fields and the Cuntz-Krieger Uniqueness
Theorem given in \cite[Theorem 6.5]{T} for graphs in which every
cycle has an exit, the so-called Condition (L).


\section{Preliminaries}

A (directed) {\it graph} $E=(E^{0},E^{1},r,s)$ consists of two sets $E^{0}$ and $E^{1}$ together with maps $r,s:E^{1}\rightarrow E^{0}$. The elements of $E^{0}$ are called \textit{vertices} and the elements of $E^{1}$ \textit{edges}. If a vertex $v$ emits no edges, that is, if $s^{-1}(v)$ is empty, then $v$
is called a \textit{sink}. A vertex $v$ is called a \textit{regular vertex} if $s^{-1}(v)$ is a finite non-empty set. The set of regular vertices is denoted by $E^0_{\text{reg}}$.

A path $\mu $ in a graph $E$ is a finite sequence of edges $\mu
=e_{1}\dots e_{n}$ such that $r(e_{i})=s(e_{i+1})$ for $i=1,\dots
,n-1$. In this case, $n=l(\mu)$ is the length of $\mu $; we view the
elements of $E^{0}$ as paths of length $0$. For any $n\in {\mathbb
N}$ the set of paths of length $n$ is denoted by $E^n$. Also,
$\text{Path}(E)$ stands for the set of all paths, i.e.,
$\text{Path}(E)=\bigcup_{n\in {\mathbb N \cup \{0\}}} E^n$. We
denote by $\mu ^{0}$ the set of the vertices of the path $\mu $,
that is, the set $\{s(e_{1}),r(e_{1}),\dots ,r(e_{n})\}$.

A path $\mu $ $=e_{1}\dots e_{n}$ is \textit{closed} if $r(e_{n})=s(e_{1})$, in which case $\mu $ is said to be {\it based at the vertex} $s(e_{1})$. The closed path $\mu $ is called a \textit{cycle} if it does not pass through any of its vertices twice, that is, if $s(e_{i})\neq s(e_{j})$ for every $i\neq j$. A cycle of length one is called a \textit{loop}. An \textit{exit} for a path $\mu =e_{1}\dots e_{n}$ is an edge $e$ such that $s(e)=s(e_{i})$ for some $i$ and $e\neq e_{i}$. We say that $E$ satisfies \textit{Condition }(L) if every simple closed path in $E$ has an exit, or, equivalently, every cycle in $E$ has an exit.

Given paths $\alpha$, $\beta$, we say $\alpha \leq \beta$ if $\beta = \alpha \alpha'$, for some path $\alpha'$.

For each $e\in E^{1}$, we call $e^{\ast }$ a {\it ghost edge}. We let $r(e^{\ast}) $ denote $s(e)$, and we let $s(e^{\ast })$ denote $r(e)$.

\begin{definition}\label{Lpa} {\rm Given an arbitrary graph $E$ and a commutative ring with unit $R$, the \textit{Leavitt path algebra with coefficients in $R$}, denoted $L_{R}(E)$, is the universal $R$-algebra generated by a set $\{v:v\in E^{0}\}$ of pairwise orthogonal idempotents together with a set of variables $\{e,e^{\ast }:e\in E^{1}\}$ which satisfy the following
conditions:

(1) $s(e)e=e=er(e)$ for all $e\in E^{1}$.

(2) $r(e)e^{\ast }=e^{\ast }=e^{\ast }s(e)$\ for all $e\in E^{1}$.

(3) (The ``CK-1 relations") For all $e,f\in E^{1}$, $e^{\ast}e=r(e)$ and $
e^{\ast}f=0$ if $e\neq f$.

(4) (The ``CK-2 relations") For every regular vertex $v\in E^{0}$,
\begin{equation*}
v=\sum_{\{e\in E^{1},\ s(e)=v\}}ee^{\ast}.
\end{equation*}}
\end{definition}

An alternative definition for $L_R(E)$ can be given using the extended graph $\widehat{E}$. This graph has the same set of vertices $E^0$ and same set of edges $E^1$ together with the so-called ghost edges $e^*$ for each $e\in E^1$, whose directions are opposite to those of the corresponding $e\in E^1$. Thus, $L_R(E)$ can be defined as the usual path algebra $R\widehat{E}$ with coefficients in $R$ subject to the Cuntz-Krieger relations (3) and (4) above.

If $\mu = e_1 \dots e_n$ is a path in $E$, we write $\mu^*$ for the element $e_n^* \dots e_1^*$ of $L_{K}(E)$. With this notation it can be shown that the Leavitt path algebra $L_{R}(E)$ can be viewed as a
$$L_{R}(E) = \text{span}_{R}\{\alpha\beta^*: \alpha,\beta \in \text{Path}(E) \text{ and } r(\alpha)=r(\beta)\}$$
and $r v \neq 0$ for all $v \in Path(E)$ and all $r \in
R\setminus\{0\}$ (see Propositions 3.4 and 4.9 of \cite{T}). (The
elements of $E^0$ are viewed as paths of length $0$, so that this
set includes elements of the form $v$ with $v\in E^0$).

We can define an $R$-linear involution $x \mapsto x^\ast$ on $L_R(E)$ as follows: if $x = \sum_{i=1}^n r_i \alpha_i \beta_i^{\ast}$ then $x^{\ast} = \sum_{i=1}^n r_i \beta_i \alpha_i^{\ast}$.

If $E^0$ is finite, then $L_{R}(E)$ is unital with $\sum_{v\in E^0} v=1_{L_{R}(E)}$; otherwise, $L_{R}(E)$ is a ring with a set of local units (i.e., a set of elements $X$ such that for every finite collection $a_1,\dots,a_n\in L_R(E)$, there exists $x\in X$ such that $a_ix=a_i=xa_i$) consisting of sums of distinct vertices of the graph.

Another useful property of $L_R(E)$ is that it is a graded algebra.
Note that any ring $R$ may be viewed as a $\mathbb{Z}$-ring in the
natural way (set $R_0 = R$ and $R_n = 0$ for every $0\neq n \in
\mathbb{Z}$). Then $L_R(E)$ can be decomposed as a direct sum of
homogeneous components $L_R(E)=\bigoplus_{n\in {\mathbb Z}}
L_R(E)_n$ satisfying $L_R(E)_nL_R(E)_m\subseteq L_R(E)_{n+m}$.
Actually, $$L_R(E)_n=\text{span}_R\{pq^*: p,q\in \text{Path}(E) ,
l(p)-l(q)=nÊ \}.$$

Every element $x_n\in L_R(E)_n$ is a homogeneous element of degree $n$.

Given this background information let us start by defining the analogous relations given for a Leavitt path algebra when we consider a general $R$-algebra with involution.

\begin{definition}\label{CuntzKriegerEsystem} {\rm Let $E$ be a graph and $\mathcal{A}$ be an $R$-algebra with involution $\ast$. A \textit{Cuntz-Krieger $E$-system in $\mathcal{A}$} is a collection $\Sigma = (S_{\mu})_{\mu \in E^0 \cup E^1} \subset \mathcal{A}$ which satisfies the following relations:

(1) For all $v, w \in E^0$, $S_vS_v = S_v$ and $S_vS_w = 0$ if $v\neq w$.

(2) $S_v^{\ast} = S_v$ for all $v \in E^0$.

(3) $S_{s(e)}S_{e} = S_{e} = S_{e}S_{r(e)}$ for all $e\in E^{1}$.

(4) For all $e,f\in E^{1}$, $S_e^{\ast}S_e=S_{r(e)}$ and $S_e^{\ast}S_f=0$ if $e\neq f$.

(5) For every regular vertex $v\in E^{0}$,
\begin{equation*}
S_v=\sum_{\{e\in E^{1},\ s(e)=v\}}S_eS_e^{\ast}.
\end{equation*}}
\end{definition}

For convenience we are going to extend the Cuntz-Krieger $E$-system $\Sigma \subset \mathcal{A}$ by defining all elements $S_\mu \in \mathcal{A}$, $\mu \in \text{Path}(E)$, considering $S_\mu S_\nu = S_{\mu\nu}$ if $r(\mu) = s(\nu)$ and $S_\mu S_\nu = 0$ otherwise. So we shall abuse the notation and write it as an extended system $\Sigma = (S_{\mu})_{\mu \in \text{Path}(E)}$.

There exists a new version given for paths of the Cuntz-Krieger relation (5) defined above.

\begin{proposition}\label{Sv} Let $v \in E^0$ and $k \in \mathbb{N}$. If there are finitely many paths $\mu$ with $s(\mu) = v$ and $l(\mu) \leq k$, then
\begin{equation*}
S_v = \sum_{\{s(\mu)=v ,\ l(\mu)=k \}}S_\mu S_{\mu}^{\ast} + \sum_{\{s(\mu)=v ,\ l(\mu)< k \text{ and } r(\mu) \text{ is a sink } \}}S_\mu S_{\mu}^{\ast}.
\end{equation*}
\end{proposition}

\begin{proof} The proof follows completely \cite[Proposition 1.3]{NR}. We only have to consider that a path $e_1e_2\ldots e_n$ in \cite{NR} corresponds to the path $e_ne_{n-1}\ldots e_1$ here, that is, the edges in \cite{NR} are multiplied so that the action of $f$ precedes the action of $e$ in the product $ef$; contrary to the action we consider here. In particular when \cite{NR} refers to a source in our case it exactly corresponds to a sink.
\end{proof}

Analogously to \cite[Equation (2.1)]{T} for Leavitt path algebras, we set here the following proposition for future reference.

\begin{proposition}\label{Gsigma} Given a Cuntz-Krieger $E$-system $\Sigma = (S_{\mu})_{\mu \in \text{Path}(E)}$ in some $R$-algebra $\mathcal{A}$ with involution, the collection
$$ G(\Sigma) = \{ S_\mu S_\nu^{\ast}: \mu, \nu \in \text{Path}(E), r(\mu) = r(\nu)\}$$
satisfies the following:
\begin{itemize}
\item[(i)] given paths $\mu, \nu \in \text{Path}(E)$ with $r(\mu) = r(\nu)$, we have $(S_\mu S_\nu^{\ast})^{\ast} = S_\nu S_\mu^{\ast}$;
\item[(ii)] given four paths $\alpha, \beta, \mu, \nu \in \text{Path}(E)$ with  $r(\alpha) = r(\beta)$ and $r(\mu) = r(\nu)$ then
\begin{align*}
 (S_\alpha S_\beta^{\ast})(S_\mu S_\nu^{\ast}) & = \begin{cases}
S_{\alpha\mu'}S_\nu^{\ast} & \text{if }\ \mu=\beta \mu' \text{ for some } \mu' \in \text{Path}(E),\\
S_{\alpha}S_{\nu\beta'}^{\ast} & \text{if }\ \beta=\mu \beta' \text{ for some } \beta' \in \text{Path}(E),\\
0 & \text{otherwise. }\  \end{cases}
\end{align*}
\end{itemize}
\end{proposition}

\begin{definition} {\rm The collection $G(\Sigma)$ given in Proposition \ref{Gsigma} will be called the \textit{standard Cuntz-Krieger generator set associated with $\Sigma$}. When we refer to the Leavitt path algebra $L_R(E)$ we will denote it simply by $G_E$, i.e.,
$$G_E = \{ \mu \nu^{\ast}: \mu, \nu \in \text{Path}(E), r(\mu) = r(\nu)\}.$$
As we have mentioned before, note that we have $L_{R}(E) = \text{span}_{R}(G_E)$.}
\end{definition}

Also now we introduce a Graded Uniqueness Theorem that follows similarly to the one given in \cite[Theorem 5.3]{T}. We will use this result later.

\begin{theorem}\label{GradedUniquenessTheorem} Let $\Sigma = (S_{\mu})_{\mu \in E^0 \cup E^1}$ be a Cuntz-Krieger $E$-system in a $\mathbb{Z}$-graded $R$-algebra $\mathcal{A}$ such that for any $\mu \in E^0$, $S_\mu$ is homogeneous of degree zero and for any $\mu \in E^1$, $S_\mu$ is homogeneous of degree one. Then the following conditions are equivalent:
\begin{itemize}
\item[(i)] the associated representation $\Phi : L_R(E) \rightarrow \mathcal{A}$ is injective;
\item[(ii)] $\Phi(rv)\neq 0$ for all $v \in E^0$ and for all $r \in R\setminus \{0\}$.
\end{itemize}
\end{theorem}


\section{The essentially aperiodic representation}

In this section we give a particular representation of the Leavitt
path algebra, which will be called the ``essentially aperiodic
representation". Previously we define all the terminology we need
for this purpose. These definitions are the analog (in our graph
sense) of those given in \cite[Definitions 2.4-2.5]{NR}.

\begin{definition} {\rm Given a graph $E$, a \textit{trail} in $E$ is either
\begin{itemize}
\item[(i)] a finite path $\tau=e_1 \ldots e_n$ (possibly of length zero) whose range $r(\tau)=r(e_n)$ is a sink, that is,
$s^{-1}(r(\tau)) = \emptyset$ ; or
\item[(ii)] an infinite path, that is, an infinite sequence $\tau = e_1 e_2 \ldots $ of edges, such
that $r(e_n)=s(e_{n+1})$ for every $n \in \mathbb{N}$.
\end{itemize}}
\end{definition}

In some literature about Leavitt path algebras (see for example \cite{AA2}), the set of all trails in the graph $E$ is denoted by $E^{\leq\infty}$.

Given a trail $\tau$ and some integer $n \geq 0$, we define the \textit{head of length} $n$ to be the finite path:
\begin{align*}
 \tau_{(n)} & = \begin{cases}
 s(\tau) & \text{if }\ n = 0,\\
 e_1 \ldots e_n & \text{if }\ n > 0 \text{ and } \tau \text{ is either infinite, or finite with } l(\tau) > n,\\
 \tau & \text{if }\ n > 0 \text{ and } \tau \text{ is finite, with } l(\tau) \leq n. \  \end{cases}
\end{align*}

Notice that all heads of $\tau$ have the same source, that is, $s(\tau_{(n)})=s(\tau_{(0)})$, and this special vertex will be denoted simply by $s(\tau)$, and will be referred to as the source of $\tau$. Given a trail $\tau$ and a path $\mu \in \text{Path}(E)$, we write $ \mu \leq \tau $ if $\mu = \tau_{(n)}$ for some $n\geq 0$. It is evident that if $ \mu \leq \tau $, then removing $\mu$ from $\tau$ still yields a trail. On the other hand, if we have trail $\tau$ and a path $\mu \in \text{Path}(E)$ with $r(\mu)=s(\tau)$, the obvious concatenation $\mu\tau$ is again a trail.

\begin{definition} {\rm Suppose $\Sigma = (S_{\mu})_{\mu \in \text{Path}(E)}$ is a Cuntz-Krieger $E$-system in an $R$-algebra $\mathcal{A}$. Consider the subset
$$G^\Delta(\Sigma)= \{S_\mu S_\mu^{\ast}: \mu \in \text{Path}(E)\}.$$
By Proposition \ref{Gsigma} it is clear that $G^\Delta(\Sigma)$ is a commutative set of self-adjoint idempotents in $<\Sigma>$, the $R$-subalgebra of $\mathcal{A}$ generated by $\Sigma$. We will refer to $G^\Delta(\Sigma)$ as the \textit{standard diagonal generator set}.
The $R$-subalgebra $<G^\Delta(\Sigma)> \ \subset \mathcal{A}$, which will be denoted by $\Delta(\Sigma)$, is called the \textit{diagonal algebra associated with $\Sigma$}.
In particular when we refer to the case of the Leavitt path algebra $L_R(E)$, the diagonal generator set will be denoted by $G_E^{\Delta}$ and the $R$-subalgebra generated by it will be denoted by $\Delta(E)$.}
\end{definition}

\begin{remark} {\rm If $\Sigma = (S_{\mu})_{\mu \in \text{Path}(E)}$ is a Cuntz-Krieger $E$-system in an $R$-algebra $\mathcal{A}$ and $\mathcal{B}$ is another $R$-algebra such that we have a $\ast$-homomorphism $\Pi : \mathcal{A} \rightarrow \mathcal{B}$, then $\Pi(\Sigma) = (\Pi(S_{\mu}))_{\mu \in \text{Path}(E)}$ is a Cuntz-Krieger $E$-system in $\mathcal{B}$ and $\Pi$ maps $G^\Delta(\Sigma)$ onto $G^\Delta(\Pi(\Sigma))$ and $\Delta(\Sigma)$ onto $\Delta(\Pi(\Sigma))$. When we specialize to the case of a representation $\Pi : L_R(E) \rightarrow \mathcal{A}$, this maps $G_E^{\Delta}$ onto $G^\Delta(\Sigma)$ and $\Delta(E)$ onto $\Delta(\Sigma)$.}
\end{remark}

For $\mu \mu^{\ast}, \nu \nu^{\ast} \in G_E^{\Delta}$ we say $\mu \mu^{\ast} \leq \nu \nu^{\ast}$ if and only if $\mu \leq \nu$, that is, $\nu = \mu \mu'$ for some $\mu' \in {\rm Path}(E)$. Then for any pair $\mu \mu^{\ast}, \nu \nu^{\ast} \in G_E^{\Delta}$ we have either $(\mu \mu^{\ast})(\nu \nu^{\ast})=0$ or $\mu \mu^{\ast}$ and $\nu \nu^{\ast}$ are comparable (either $\mu \mu^{\ast} \leq \nu \nu^{\ast}$ or $\nu \nu^{\ast} \leq \mu \mu^{\ast}$).

\begin{proposition}\label{PropMaxTotOrdSubset} Consider the diagonal generator set $G_E^{\Delta}$. Then the finite (resp. countably infinite) maximal totally ordered subsets of  $G_E^{\Delta}$ are in one-to-one correspondence with the finite (resp. infinite) trails in $E$. Specifically, every maximal totally ordered subset $\mathcal{P} \subset  G_E^{\Delta}$ can be uniquely presented as
$\mathcal{P}_{\tau} = \{ \tau_{(n)}\tau_{(n)}^{\ast}: n \geq 0 \}$ for some trail $\tau$.
\end{proposition}

\begin{proof} First let $\tau$ be a finite trail ($\tau = e_1 \ldots e_n$ and $r(\tau)$ is a sink). We claim that $\mathcal{P}= \{ \tau_{(0)}\tau_{(0)}^{\ast},\tau_{(1)}\tau_{(1)}^{\ast},\ldots,\tau_{(n)}\tau_{(n)}^{\ast} \}$ is a maximal totally ordered subset of $G_E^{\Delta}$. Obviously $\mathcal{P}$ is a totally ordered set with order $\leq$. Assume that there exists $\mathcal{P} \subsetneq \mathcal{Q}$ and $\mathcal{Q}$ is a totally ordered subset of $G_E^{\Delta}$; then there exists $\beta\beta^{\ast} \in \mathcal{Q} \setminus \mathcal{P}$ and so $\mathcal{P} \cup \{\beta \beta^{\ast}\}$ is totally ordered. Then $\beta \beta^{\ast} \leq \tau_{(n)}\tau_{(n)}^{\ast}$ or $\tau_{(n)}\tau_{(n)}^{\ast} \leq \beta\beta^{\ast}$. If $\beta \beta^{\ast} \leq \tau_{(n)}\tau_{(n)}^{\ast}$ then $\beta \leq \tau_{(n)}$ which is a contradiction since $\beta\beta^{\ast} \notin \mathcal{P}$; on the other hand if $\tau_{(n)}\tau_{(n)}^{\ast} \leq \beta\beta^{\ast}$ then $\tau_{(n)} \leq \beta$ and since $r(\tau_{(n)})=r(\tau)$ is a sink then $\tau_{(n)}=\beta$ which gives again a contradiction since $\beta\beta^{\ast} \notin \mathcal{P}$. Then $\mathcal{P}$ is maximal totally ordered subset of $G_E^{\Delta}$.

Consider now $\tau$ an infinite trail: $\tau = e_1 e_2 e_3 \ldots$. Let $\mathcal{P}=\{ \tau_{(0)}\tau_{(0)}^{\ast}, \tau_{(1)}\tau_{(1)}^{\ast}, \tau_{(2)}\tau_{(2)}^{\ast}, \ldots\}$. Let us prove that $\mathcal{P}$ is a maximal totally ordered subset of $G_E^{\Delta}$. Take into account that $\mathcal{P}$ is totally ordered set with order $\leq$. Suppose there exists a totally ordered subset $\mathcal{Q}$ of $G_E^{\Delta}$ such that $\mathcal{P} \subsetneq \mathcal{Q}$ and consider $\beta \beta^{\ast} \in \mathcal{Q} \setminus \mathcal{P}$. Then again $\mathcal{P} \cup \{\beta\beta^{\ast}\}$ is totally ordered. If there exists $n$ such that $\beta \beta^{\ast} \leq \tau_{(n)}\tau_{(n)}^{\ast}$ then $\beta \leq \tau_{(n)}$ which is a contradiction. Therefore for any $n$, $\tau_{(n)}\tau_{(n)}^{\ast} \leq \beta\beta^{\ast}$ and then $\tau_{(n)} \leq \beta$ which is impossible since $\beta \beta^{\ast} \in G_E^{\Delta}$. So $\mathcal{P}$ is a maximal totally ordered subset of $G_E^{\Delta}$.

Conversely let $\mathcal{P}$ be a finite maximal totally ordered subset of $G_E^{\Delta}$. Then by using Zorn's Lemma we can find a maximal element in $\mathcal{P}$. Assume that $\beta \beta^{\ast}$ is a maximal element of $\mathcal{P}$. Then $\beta$ is a finite path and since $\mathcal{P}$ is maximal totally ordered subset of $G_E^{\Delta}$ then $r(\beta)$ is a sink giving that $\beta$ is a finite trail.

Finally suppose that $\mathcal{P}$ is a countable infinite subset of $G_E^{\Delta}$. Then by Zorn's Lemma $\mathcal{P}$ has a minimal element $e_0e_0^{\ast}$. Also let $e_1e_1^{\ast}$ the minimal element of $\mathcal{P} \setminus \{e_0e_0^{\ast}\}$, $e_2e_2^{\ast}$ the minimal element of $\mathcal{P} \setminus \{e_0e_0^{\ast},e_1e_1^{\ast}\}$ etc. Proceeding in this way we have that $e_0e_1e_2 \ldots$ is an infinite trail.

Therefore we can say that for any trail there exists a maximal totally ordered subset of $G_E^{\Delta}$ and for any finite or infinite countable maximal totally ordered subset of $G_E^{\Delta}$ there exists a trail. In the end if $E$ is countable, then the maximal totally ordered subsets of $G_E^{\Delta}$ are in one-to-one correspondence with the trails in $E$.
\end{proof}

\begin{definition}\label{EssentiallyAperiodicTrail} {\rm Let $E$ be a graph.
\begin{itemize}
\item[(A)] An infinite trail $\tau = e_1 e_2 \ldots$ in $E$ is said to be \textit{periodic}, if there exist integers $j,k \geq 1$, such that $e_{n+k}= e_{n}$ for every $n \geq j$.
In this case, it is clear that the path $\rho = e_j \ldots e_{j+k-1}$ is closed. If we take $j$ and $k$ such that $j+k$ is the smallest possible value which satisfies the condition $e_{n+k}= e_{n}$ for every $n \geq j$ and we consider the paths $\alpha = e_1 \ldots e_{j-1}$ and $\lambda = e_j \ldots e_{j+k-1}$, we will refer to the pair $(\alpha, \lambda)$ as the \textit{seed of} $\tau$. Of course $\alpha$ may have length zero. In any case, $\lambda$ is a closed path, which will be called the \emph{period of} $\tau$.

\item[(B)] A trail $\tau$ is said to be \textit{essentially aperiodic} if:
\begin{itemize}
\item[(i)] $\tau$ is finite, or

\item[(ii)] $\tau$ is periodic and its period is a closed path without exits (which means that it has to be a cycle without exits), or

\item[(iii)] $\tau$ is infinite and not periodic.
\end{itemize}
The trails of the form (i) and (ii) will be called \textit{discrete}, while the ones of type (iii) will be called
\textit{continuous}. In a discrete essentially aperiodic trail $\tau$, we refer to a certain path as the \textit{essential head of} $\tau$, and denoted by $\tau_{{\rm(ess)}}$, which is defined accordingly as
follows:
\begin{itemize}
\item[(i)] if $\tau$ is finite, then $\tau_{{\rm(ess)}} = \tau$;

\item[(ii)] if $\tau$ is periodic, with seed $(\alpha, \lambda)$ then $\tau_{{\rm(ess)}} = \alpha$.
\end{itemize}
\end{itemize}
Let us denote by $\mathfrak{T}_E$ the set of all essentially aperiodic trails in $E$.}
\end{definition}

\begin{remark}{ \rm Note that in the Definition \ref{EssentiallyAperiodicTrail} $(A)$, the choice of $j$ and $k$ such that $j+k$ is the smallest possible value which satisfies the condition $e_{n+k}= e_{n}$ for every $n \geq j$, is unique. Assume that there exist $j_1,k_1,j_2,k_2 \geq 1$ such that $j_1+k_1=j_2+k_2$ is the smallest possible value which satisfies the condition $e_{m+k_1} = e_m$ for every $m \geq j_1$ and $e_{n+k_2} = e_n$ for every $n \geq j_2$. Then for every $r \geq 0$, $e_{j_1+r}=e_{j_1+r+k_1}=e_{j_2+r+k_2}=e_{j_2+r}$ (since $j_1+k_1=j_2+k_2$). Suppose that $j_2 > j_1$ (the case $j_1>j_2$ is similar). So for any $r \geq 0$, $e_{j_1+j_2-j_1+r}=e_{j_2+r}=e_{j_1+r}$ and hence for any $m\geq j_1$, $e_{m+j_2-j_1}=e_m$. By minimality of $k_2+j_2$ we have that $k_2=0$ which is a contradiction. Thus $j_1$ and $k_1$ are unique satisfying that $j_1+k_1$ is the smallest value with this desired property.}
\end{remark}

\begin{lemma} \label{LemmaAperiodicTrail} For every vertex $v \in E^0$ there exists at least one essentially aperiodic trail $\tau$ with $v = s(\tau)$.
\end{lemma}

\begin{proof} The proof follows completely \cite[Lemma 2.6]{NR}. We just need to take into account that the action in the graph is changed.
\end{proof}

We are now in the position to give the so-called ``essentially aperiodic representation" for a Leavitt path algebra.

\begin{proposition}\label{Piaprepresentation} Let $E$ be a graph and $R$ be a commutative ring with unit. If we consider $M$ the $R$-module generated by a set of generators $\{\xi_{\tau}^n: n \in \mathbb{Z}, \tau \in \mathfrak{T}_E\}$ then there exist a $\mathbb{Z}$-graded $R$-algebra $\mathcal{E} \subseteq {\rm End}_R(M)$ and a unique Cuntz-Krieger $E$-system $\Sigma = (S_{\alpha})_{\alpha \in \text{Path}(E)} \subset \mathcal{E}$ such that the map $S_{\alpha}: M \rightarrow M$ is given by:
\begin{align*}
S_{\alpha}(\xi_{\tau}^n)  & = \begin{cases}
\xi_{\alpha\tau}^{l(\alpha) + n}  & \text{if }\ r(\alpha) = s(\tau),\\
0 & \text{otherwise. }\  \end{cases}
\end{align*}
Besides, the associated representation $\Pi_{\rm ap}: L_R(E)
\rightarrow \mathcal{E}$ is injective.
\end{proposition}

\begin{definition} The representation $\Pi_{\rm ap}$ given in Proposition \ref{Piaprepresentation} is called the \emph{essentially aperiodic
representation} of $L_R(E)$.
\end{definition}

\begin{proof}[Proof of Proposition~\ref{Piaprepresentation}]
First let us define for each $i \in \mathbb{Z}$, $M_i$ as the
$R$-module generated by the collection $\{\xi_{\tau}^i: \tau \in
\mathfrak{T}_E\}$. Consider $M = \bigoplus_{i \in \mathbb{Z}}M_i$ as
$\mathbb{Z}$-graded $R$-module. Let $f \in {\rm End}_R(M)$, we say
that $f$ is \textit{homogeneous of degree $k$} ($k \in \mathbb{Z}$)
if $f(M_i) \subseteq M_{i+k}$ for every $i$. In this case we set
${\rm deg}(f) = k$. Then define $\mathcal{E}_i = \{f \ | \ f \in
{\rm End}_R(M), {\rm deg}(f) = i \}$ and let $\mathcal{E}=
\bigoplus_{i \in \mathbb{Z}}\mathcal{E}_i$ be an $R$-subalgebra of
${\rm End}_R(M)$.

Let us check that $\mathcal{E}$ is $\mathbb{Z}$-graded,  that is, for each $i,j \in \mathbb{Z}$, $\mathcal{E}_i\mathcal{E}_j \subseteq \mathcal{E}_{i+j}$. Suppose $f \in \mathcal{E}_i$ and $g \in \mathcal{E}_j$ and consider $f\circ g$. Take some $t \in \mathbb{Z}$, and then $(f \circ g) (M_t) = f(g(M_t))$. Since ${\rm deg}(g) = j$, $g(M_t) \subseteq M_{t+j}$. Then taking into account that ${\rm deg}(f) = i$, we have
$f(g(M_t)) \subseteq f(M_{t+j}) \subseteq M_{t+j+i}$ which means $f \circ g \in \mathcal{E}_{i+j}$.

Now let $\alpha \in \text{Path}(E)$ and $S_{\alpha}: M \rightarrow M$ be an $R$-homomorphism which is defined on the generators of $M$ as:
 \begin{align*}
S_{\alpha}(\xi_{\tau}^n)  & = \begin{cases}
\xi_{\alpha\tau}^{l(\alpha) + n}  & \text{if }\ r(\alpha) = s(\tau),\\
0 & \text{otherwise. }\  \end{cases}
\end{align*}

Let us prove that $S_\alpha$ is homogeneous of degree $l(\alpha)$, that is, for any $t \in \mathbb{Z}$ we have $S_{\alpha}(M_t) \subseteq M_{t+l(\alpha)}$. So consider $\tau \in \mathfrak{T}_E$ and let $\xi_{\tau}^t \in M_t$. Then
\begin{align*}
S_{\alpha}(\xi_{\tau}^t)  & = \begin{cases}
\xi_{\alpha\tau}^{l(\alpha) + t}  & \text{if }\ r(\alpha) = s(\tau),\\
0 & \text{otherwise. }\  \end{cases}
\end{align*}
and therefore $S_{\alpha}(\xi_{\tau}^t) \in M_{t+l(\alpha)}$ obtaining the desired statement.

On the other hand, let us define for each $\alpha \in
\text{Path}(E)$, $S_\alpha^{\ast} = S_{\alpha^{\ast}}$, where
$\alpha^{\ast}$ is the ghost path of $\alpha$ and
\begin{align*}
S_{\alpha^\ast}(\xi_{\tau}^n)  & = \begin{cases}
\xi_{\beta}^{n - l(\alpha)}  & \text{if }\ \alpha \leq \tau \text{ with } \tau = \alpha \beta ,\\
0 & \text{otherwise. }\  \end{cases}
\end{align*}
A similar argument shows that $S_\alpha^{\ast}$ is homogeneous of degree $-l(\alpha)$.

Let us prove that $\Sigma = (S_{\alpha})_{\alpha \in \text{Path}(E)}$ is a Cuntz-Krieger $E$-system inside $\mathcal{E}$. Consider $\mu, \nu \in \text{Path}(E)$ and $n \in \mathbb{Z}$, $\tau \in \mathfrak{T}_E$ then on the one hand
\begin{align*}
S_{\mu}S_\nu(\xi_{\tau}^n)  & = \begin{cases}
S_\mu (\xi_{\nu\tau}^{l(\nu) + n})  & \text{if }\ r(\nu) = s(\tau),\\
0 & \text{otherwise. }\  \end{cases}
\end{align*}
\begin{align*}
 & = \begin{cases}
\xi_{\mu\nu\tau}^{l(\mu)+ l(\nu) + n}  & \text{if }\ r(\mu) = s(\nu\tau) \text{ and } r(\nu) = s(\tau),\\
0 & \text{otherwise. }\  \end{cases}
\end{align*}
On the other hand,
\begin{align*}
 S_{\mu\nu}(\xi_{\tau}^n) & = \begin{cases}
\xi_{\mu\nu\tau}^{l(\mu)+ l(\nu) + n}  & \text{if }\ r(\mu\nu) = s(\tau),\\
0 & \text{otherwise. }\  \end{cases}
\end{align*}
So
\begin{align*}
S_{\mu}S_\nu & = \begin{cases}
S_{\mu\nu}   & \text{if }\ r(\mu) = s(\nu),\\
0 & \text{otherwise. }\  \end{cases}
\end{align*}

Similarly to this case also it can be proved that
\begin{align*}
S_{\mu}S_\nu^{\ast} & = \begin{cases}
S_{\mu{\nu}^{\ast}}   & \text{if }\ r(\mu) = r(\nu),\\
0 & \text{otherwise }\  \end{cases}
\end{align*}
and
\begin{align*}
S_{\mu}^{\ast}S_\nu & = \begin{cases}
S_{{\mu}^{\ast}\nu}   & \text{if }\ s(\mu) = s(\nu),\\
0 & \text{otherwise. }\  \end{cases}
\end{align*}

Considering $\mu,\nu$ also as vertices (paths of length zero), conditions (1) and (3) from Definition \ref{CuntzKriegerEsystem} are satisfied. Condition (2) follows from the fact that for any vertex $v$, $S_v^{\ast} = S_{v^{\ast}}=S_v$. Again for any edges $e,f \in E^1$,  $S_e^{\ast}S_e=S_{e^{\ast}e}=S_{r(e)}$ and if $e\neq f$, $S_e^{\ast}S_f=S_{e^{\ast}f}=0$ having condition (4). It remains to prove equation (5) from Definition \ref{CuntzKriegerEsystem}: consider $v$ a regular vertex then for any $n \in \mathbb{Z}$, and $\tau \in \mathfrak{T}_E$,
\begin{align*}
 \sum_{\{e\in E^{1},\ s(e)=v\}}S_eS_e^{\ast}(\xi_{\tau}^n) & = \begin{cases}
\xi_{\tau}^n  & \text{if }\ e \leq \tau \text{ for some } e \in s^{-1}(v),\\
0 & \text{otherwise. }\  \end{cases}
\end{align*}
and also
\begin{align*}
S_{v}(\xi_{\tau}^n)  & = \begin{cases}
\xi_{\tau}^n  & \text{if }\ v = s(\tau),\\
0 & \text{otherwise. }\  \end{cases}
\end{align*}
which means $(S_v - \sum_{\{e\in E^{1},\ s(e)=v\}}S_eS_e^{\ast})(\xi_{\tau}^n) = 0$ obtaining condition (5).

Let us define now the representation $\Pi_{\rm ap}: L_R(E) \rightarrow \mathcal{E}$. We know that every $x \in L_R(E)$ can be written as $x = \sum_{i=1}^n r_i \alpha_i \beta_i^{\ast}$ where $r_i \in R$, $\alpha_i, \beta_i \in \text{Path}(E)$ for $i=1,\ldots,n$. For every $\alpha\beta^{\ast}$ of the previous form let $\Pi_{\rm ap}(\alpha\beta^{\ast}) =  S_{\alpha\beta^{\ast}}$ and extend the definition $R$-linearly for every $x \in L_R(E)$. It is a straightforward computation to see that $\Pi_{\rm ap}$ is a well-defined homomorphism.

To prove that $\Pi_{\rm ap}$ is $\mathbb{Z}$-graded homomorphism, consider $n \in \mathbb{Z}$ and $\alpha\beta^{\ast} \in {L_R(E)}_n$, which means $l(\alpha)-l(\beta) = n$. Then $\Pi_{\rm ap}(\alpha\beta^{\ast}) =  S_{\alpha\beta^{\ast}} = S_\alpha S_\beta^{\ast} \in \mathcal{E}_{l(\alpha)}\mathcal{E}_{-l(\beta)} \subseteq \mathcal{E}_{l(\alpha)-l(\beta)} = \mathcal{E}_{n}$. Hence we obtain,
$$\Pi_{\rm ap}({L_R(E)}_n) \subseteq \mathcal{E}_{n}.$$

Finally let us show that $\Pi_{\rm ap}$ is injective. By Lemma \ref{LemmaAperiodicTrail}, for every vertex $v \in E^0$ there exists at least one essentially aperiodic trail $\tau$ with $v = s(\tau)$. Then $S_v(\xi_{\tau}^n)= \xi_{\tau}^n$ for every $n \in \mathbb{Z}$ implying that $\Pi_{\rm ap}(rv)\neq 0$ for all $r \in R\setminus\{0\}$. Apply Theorem \ref{GradedUniquenessTheorem} to obtain the injectivity of the essentially aperiodic representation.
\end{proof}


\section{The commutative core}

In this section we describe a commutative subalgebra inside $L_R(E)$, the so-called ``commutative core". Previously we need a few definitions and a lemma that will be used in Proposition \ref{CommutativeSubalgebra}.

\begin{definition} \label{Defdistinguished} {\rm For any infinite discrete essentially aperiodic trail (Definition \ref{EssentiallyAperiodicTrail}(B)(ii)) which is parameterized by the seed $(\alpha,\lambda_\alpha)$ of the trail (that is, $\alpha  \in \text{Path}(E)$ is its essential head and $r(\alpha)$ is visited by the cycle without exits $\lambda_{\alpha}$), the path $\alpha$ will be called a \textit{distinguished} path. In the case $l(\alpha) = 0$, we will call it a \textit{distinguished} vertex.}
\end{definition}

\begin{remark} {\rm Consider a distinguished path $\alpha$. This means that if $l(\alpha) = 0$ then $\alpha$ is simply a vertex $v$ and the cycle $\lambda_{\alpha}$ starts and ends at $v$. In the case $l(\alpha) \geq 1$, suppose $\alpha = e_1\ldots e_n$, then $\lambda_{\alpha}$ is a cycle that starts and ends at $r(\alpha)=r(e_n)$ but does not visit any of the vertices $s(e_1),\ldots,s(e_n)$.
Of course, for any distinguished path $\alpha$, $r(\alpha)$ is a distinguished vertex.}
\end{remark}

\begin{lemma}\label{CycleNoExits} A cycle $\lambda$ has no exits if and only if $\lambda \lambda^{\ast} = s(\lambda)$.
\end{lemma}

\begin{proof} The result is proved similarly to \cite[Lemma 3.2]{NR}: following it we only need to use Proposition \ref{Sv} instead, since the action in the graph in \cite{NR} is changed as we have commented before; in particular, the source and the range maps in \cite{NR} are respectively the range and the source maps here.
\end{proof}

\begin{definition} {\rm An element $x \in L_R(E)$ is said to be \textit{normal} if $xx^{\ast}=x^{\ast}x$.}
\end{definition}

\begin{proposition}\label{CommutativeSubalgebra} Let $\alpha, \beta \in \text{Path}(E)$ with $r(\alpha) = r(\beta)$. The generator $\alpha\beta^{\ast} \in G_E$ is a normal element in $L_R(E)$ if and only if one of the following holds:
\begin{enumerate}
\item $\alpha = \beta$;
\item $\beta \leq \alpha$ and $\beta$ is a distinguished path, i.e. $\alpha = \beta \lambda_{\beta}$ and $\lambda_{\beta}$ is a cycle without exits;
\item $\alpha \leq \beta$ and $\alpha$ is a distinguished path, i.e. $\beta = \alpha \lambda_{\alpha}$ and $\lambda_{\alpha}$ is a cycle without exits.
\end{enumerate}
Also, if we denote by $G_E^M$ the set of all such normal generators, then the $R$-algebra  $M_R(E)$  generated by $G_E^M$,  $M_R(E) = \ <G_E^M> \ \subseteq L_R(E)$ is commutative.
\end{proposition}

\begin{definition} The subalgebra $M_R(E)$ given in Proposition \ref{CommutativeSubalgebra} is called the \emph{commutative core} of $L_R(E)$.
\end{definition}

\begin{proof}[Proof of Proposition~\ref{CommutativeSubalgebra}] The necessary and sufficient condition for normality is proved in a similar way by following the ideas given in \cite[Proposition-Definition 3.1]{NR}. We will include here the proof for completeness. Consider first $x = \alpha\beta^{\ast}$ satisfying one of the conditions (1), (2) or (3). If $\alpha = \beta$ then $x$ is clearly normal. Suppose $\beta \leq \alpha$ (the case $\alpha \leq \beta$ is done similarly). We have $\alpha = \beta \lambda$ and $\lambda$ is a cycle without exits; by Lemma \ref{CycleNoExits} $\lambda \lambda^{\ast} = s(\lambda) = r(\beta)$ so
$$x x^{\ast}=  \beta \lambda \beta^{\ast} \beta\lambda^{\ast} \beta^{\ast} = \beta \lambda \lambda^{\ast} \beta^{\ast} = \beta \beta^{\ast} = \beta\alpha^{\ast}\alpha \beta^{\ast} = x^{\ast} x.$$

For the converse implication suppose now that $x = \alpha\beta^{\ast}$ is a non-zero normal element of $L_R(E)$. Since $x\neq 0 $ then $r(\alpha) = r(\beta)$ and since $x x^{\ast} =x^{\ast} x$ then $s(\alpha) = s(\beta)$. We have $(x x^{\ast})^2 \neq 0$ but $(x x^{\ast})^2= x^{\ast}x x x^{\ast}$ so $x^2$ is non-zero. But having $x^2 = (\alpha\beta^{\ast})(\alpha\beta^{\ast}) \neq 0$ implies that $\beta \leq \alpha$ or $\alpha \leq \beta$ by Proposition \ref{Gsigma}. We can assume that $\beta \leq \alpha$. Considering all these facts together necessarily $\alpha = \beta \lambda$ where $\lambda$ is a closed  path. If we suppose that $\lambda$ has an exit then by Lemma \ref{CycleNoExits}, $\lambda \lambda^{\ast} \neq s(\lambda) = r(\beta)$ and finally
 $$x x^{\ast} = \beta \lambda \lambda^{\ast} \beta^{\ast} \neq \beta \beta^{\ast} = x^{\ast} x,$$
 which is a contradiction.

In order to prove the commutativity of  $M_R(E)$, we are going to check the following points:
\begin{itemize}
\item[(i)] Elements of the form (1) commute: this is clear since $\alpha = \beta$ then $\alpha \alpha^{\ast} \in G_E^\Delta$.

\item[(ii)] Elements of the form (1) commute with elements of the form (2): let $x = \alpha \alpha^{\ast}$ of the form (1) and $y = \mu\nu^{\ast}$ of the form (2). We know that $\nu \leq \mu$ and $\mu = \nu \lambda_{\nu}$ where $\lambda_{\nu}$ is a cycle without exits. Let us prove that $xy = yx$. First we show that
$$xy = 0 \Leftrightarrow yx = 0.$$
By Proposition \ref{Gsigma} we have
\begin{align*}
yx= \mu \nu^{\ast} \alpha \alpha^{\ast} & = \begin{cases}
\mu\alpha'\alpha^{\ast} & \text{if }\ \alpha=\nu \alpha' \text{ for some } \alpha' \in \text{Path}(E),\\
\mu(\alpha\nu')^{\ast} & \text{if }\ \nu=\alpha \nu' \text{ for some } \nu' \in \text{Path}(E),\\
0 & \text{otherwise. }\  \end{cases}
\end{align*}

If $\alpha=\nu \alpha'$ for some $\alpha' \in \text{Path}(E)$, $yx= \nu \lambda_\nu \alpha' \alpha^{\ast}$. We know $\lambda_\nu$ is a cycle without exits so then $\alpha'$ is either a subpath of the cycle $\lambda_\nu$, i.e., $\alpha' \leq \lambda_\nu$, or $\alpha' = \lambda_\nu^t$ for some $t \in \mathbb{N}$. Then $xy= \nu \alpha' \alpha'^{\ast}\nu^{\ast}\nu \lambda_\nu \nu^{\ast} = \nu \alpha' \alpha'^{\ast}\lambda_\nu \nu^{\ast} \neq 0$.

If $\nu = \alpha \nu'$ then $yx = \mu {\nu'}^{\ast} \alpha^{\ast} = \nu \lambda_\nu {\nu'}^{\ast} \alpha^{\ast}$. But $\lambda_\nu$ is a cycle without exits so $\nu'$ is either $\nu' \leq \lambda_\nu$ or $\nu' = \lambda_\nu^t$ for some $t$. Then $xy = \alpha \alpha^{\ast}\mu \nu^{\ast}= \alpha \alpha^{\ast}\nu \lambda_\nu {\nu'}^{\ast} \alpha^{\ast} = \alpha \alpha^{\ast} \alpha \nu' \lambda_\nu {\nu'}^{\ast} \alpha^{\ast} = \alpha \nu' \lambda_\nu {\nu'}^{\ast} \alpha^{\ast} \neq 0$.

So if $yx \neq 0$ then $xy \neq 0$, which is equivalent to say that if $xy=0$ we have $yx=0$. The case $yx=0$ implies $xy=0$ can be done similarly.

Now by Proposition \ref{Gsigma} the condition $xy \neq 0$ (which implies $yx \neq 0$) holds if and only if $\alpha=\mu\alpha'$ for some $\alpha'\in \text{Path}(E)$ or $\mu=\alpha\mu'$ for some $\mu'\in \text{Path}(E)$.

Suppose that $\alpha = \mu \alpha'$ then $xy= \alpha \alpha^{\ast}\mu \nu^{\ast}= \alpha {\alpha'}^{\ast} \mu^{\ast} \mu \nu^{\ast} = \alpha {\alpha'}^{\ast} \nu^{\ast} $ and $\alpha=\nu \lambda_\nu \alpha'$. We have that $\lambda_\nu$ is a cycle without exits, so $\alpha = \nu \lambda_\nu^t \lambda'$ where $t \in \mathbb{N}$ and $\lambda' \leq \lambda_\nu$. Then $xy=\nu \lambda_\nu^t \lambda' (\lambda_\nu^{t-1}\lambda')^{\ast}\nu^{\ast}= \nu \lambda_\nu^t \lambda' \lambda'^{\ast} (\lambda_\nu^{t-1})^{\ast}\nu^{\ast}$. Now notice that $\lambda'\lambda'^{\ast}=s(\lambda_\nu)$: we know that $\lambda_\nu$ is a cycle without exits and $\lambda' \leq \lambda_\nu$ so $\lambda_\nu = \lambda' \lambda''$ for certain $\lambda'' \in \text{Path}(E)$; by Lemma \ref{CycleNoExits} $s(\lambda_\nu)= \lambda_\nu \lambda_\nu^{\ast} = \lambda' \lambda''(\lambda' \lambda'')^{\ast} = \lambda' \lambda''\lambda''^{\ast} \lambda'^{\ast}$ and consequently $s(\lambda_\nu)= \lambda'^{\ast} s(\lambda_\nu)\lambda'= \lambda'' \lambda''^{\ast}$ which in the end gives us $s(\lambda_\nu)= \lambda' (\lambda''\lambda''^{\ast}) \lambda'^{\ast}= \lambda' \lambda'^{\ast}$. So then $xy= \nu \lambda_\nu^t (\lambda_\nu^{t-1})^{\ast} \nu^{\ast}= \nu \lambda_\nu \nu^{\ast}$ by using Lemma \ref{CycleNoExits} again. On the other hand, $yx=\mu \nu^{\ast}\alpha \alpha^{\ast}= \nu \lambda_\nu \nu^{\ast} \nu \lambda_\nu^t \lambda' {\lambda'}^{\ast} (\lambda_\nu^t)^{\ast} \nu^{\ast}$. Again we have that $\lambda_\nu^t \lambda' {\lambda'}^{\ast} (\lambda_\nu^t)^{\ast} = s(\lambda_\nu)$ and finally
$$yx= \nu \lambda_\nu \nu^{\ast} \nu \nu^{\ast} = \nu \lambda_\nu \nu^{\ast}=xy.$$

Now assume that $\mu=\alpha\mu'$ for some $\mu'\in \text{Path}(E)$, then by Proposition \ref{Gsigma}, $xy=\alpha\mu'\nu^{\ast}=\mu\nu^{\ast}$. Since $yx \neq 0$, then by Proposition \ref{Gsigma}, $\alpha=\nu\alpha''$ for some $\alpha''\in \text{Path}(E)$ or $\nu=\alpha\nu'$ for some $\nu'\in \text{Path}(E)$. In case $\alpha=\nu\alpha''$, $yx=\mu\alpha''\alpha^{\ast}$. Now $\lambda_{\nu}$ is a cycle without exits so $\alpha'' = \lambda_\nu^t \lambda'$ where $t \in \mathbb{N}$ and $\lambda' \leq \lambda_\nu$; a similar argument to the previous paragraph yields to $xy=yx$. If $\nu=\alpha\nu'$, then by Proposition \ref{Gsigma}, $yx=\mu(\alpha\nu')^{\ast}=\mu\nu^{\ast}=xy$. In the end we have $xy=yx$ as desired.

\item[(iii)] Elements of the form (2) commute: consider  both $x = \alpha \beta^{\ast}$ and $y = \mu\nu^{\ast}$ of the form (2), that is,  $\beta \leq \alpha$ with $\alpha = \beta \lambda_{\beta}$ and $\nu \leq \mu$ with $\mu = \nu \lambda_{\nu}$ where $\lambda_{\beta}$ and $\lambda_{\nu}$ are cycles without exits. Then arguing similarly to case (ii) we have that
$$xy = 0 \Leftrightarrow yx = 0.$$

By Proposition \ref{Gsigma} $xy \neq 0$ (which implies $yx \neq 0$) if $\beta \leq \nu$ or $\nu \leq \beta$; but except when $\beta = \nu$, the remaining options are not possible since $\beta$ and $\nu$ are distinguished paths. In this case then $xy = yx$.
\end{itemize}

Analogously to the previous steps it can be proved that:
\begin{itemize}
\item[(ii)'] elements of the form (1) commute with elements of the form (3);
\item[(iii)'] elements of the form (2) commute with elements of the form (3) and
\item[(iii)''] elements of the form (3) commute.
\end{itemize}
Finally we get that $M_R(E)$ is commutative.
\end{proof}

\begin{remark}\label{Omega_alpha}{\rm For a distinguished path $\alpha$ let $\omega_\alpha = \alpha \lambda_\alpha \alpha^{\ast}$, where $\lambda_\alpha$ is the cycle without exits that starts and ends at $s(\alpha)$. We know by Proposition \ref{CommutativeSubalgebra} that $\omega_\alpha \in G_E^M$. It so happens that the $R$-algebra $< \omega_\alpha > \  \subseteq M_R(E)$ generated by $\omega_\alpha$ is unital (with unit $\alpha \alpha^{\ast}$) and $\omega_\alpha$ is invertible in $< \omega_\alpha >$:
$$\omega_\alpha^{\ast}\omega_\alpha = \omega_\alpha \omega_\alpha^{\ast} = \alpha \alpha^{\ast}.$$

We can define then the powers $\omega_\alpha^n$ for every $n \in \mathbb{Z}$: if $n < 0$ we let $\omega_\alpha^n := (\omega_\alpha^{\ast})^{-n}$ and $\omega_\alpha^0 :=\alpha\alpha^{\ast}$. In this case it follows that there exists a unique $\ast$-isomorphism $\Gamma_\alpha:  \ < \omega_\alpha > \rightarrow R[x,x^{-1}]$ by simply defining $\Gamma_\alpha(\alpha\alpha^{\ast}) = 1$, $\Gamma_\alpha(\omega_\alpha) = x$ and $\Gamma_\alpha(\omega_\alpha^{\ast}) = x^{-1}$.

Notice that we can write $G_E^M$ as a disjoint union
$$G_E^M = G_E^\Delta \cup \{ \omega_\alpha^n : \alpha \text{ is a distinguished path, } n \neq 0\}.$$}
\end{remark}

The following definition is an important tool in order to prove that $M_R(E)$ is a maximal commutative subalgebra of $L_R(E)$. This will be the next goal of this section.

\begin{definition} {\rm Let $\mathcal{A}$ be an $R$-algebra and $\mathcal{B} \subseteq \mathcal{A}$ an $R$-subalgebra. A linear map $\mathbb{E}: \mathcal{A} \rightarrow \mathcal{B}$ is called an \textit{algebraic conditional expectation of $\mathcal{A}$ onto $\mathcal{B}$} if it satisfies the following conditions:
\begin{itemize}
\item[(i)] $\mathbb{E}$ is idempotent and ${\rm Im}(\mathbb{E}) = \mathcal{B}$.
\item[(ii)] For every $a \in \mathcal{A}$ and $b \in \mathcal{B}$, $\mathbb{E}(ba) = b\mathbb{E}(a)$ and $\mathbb{E}(ab) = \mathbb{E}(a)b$.
\end{itemize}}
\end{definition}

As far as the authors know, the idea of an algebraic conditional expectation has not appeared in the literature anywhere before.

For our purposes we need to find an appropriate algebraic conditional expectation of the $R$-algebra $\mathcal{E}$ defined in Proposition \ref{Piaprepresentation}. Consider $M$ the $R$-module generated by the set $\{\xi_{\tau}^n: n \in \mathbb{Z}, \tau \in \mathfrak{T}_E\}$ as in Proposition \ref{Piaprepresentation}. Given $\tau \in \mathfrak{T}_E$ an essentially aperiodic trail, first define the projection $Q_{\tau} \in \mathcal{E}$, $Q_{\tau}: M \rightarrow M$ such that: for every $\tau' \in \mathfrak{T}_E$ and every $n \in \mathbb{Z}$,
\begin{align*}
Q_\tau(\xi_{\tau'}^n)  & = \begin{cases}
\xi_{\tau}^n & \text{if }\ \tau' = \tau, \\
0 & \text{otherwise. }\  \end{cases}
\end{align*}

\begin{remark}{\rm It is straightforward that $Q_\tau$ is idempotent and all elements in the collection $(Q_\tau)_{\tau \in \mathfrak{T}_E}$ are mutually orthogonal. Furthermore it satisfies that, for any $m \in M$, there exists a unique minimal finite subset $\{\tau_1,\ldots,\tau_n\}$ of $\mathfrak{T}_E$ such that
$$(\sum_{i=1}^n Q_{\tau_i})(m) = m.$$ In fact if
$m=\sum_{i,j}r_{ij}\xi_{\tau_i}^{t_j}$ with $1\leq i\leq n, 1\leq
j\leq m$, $r_{ij}\in R$, $\tau_i\in \mathfrak{T}_E$ and $t_j\in
\mathbb{Z}$ then $(\sum_{i=1}^n Q_{\tau_i})(m) = m$ and for any
subset $\{\tau'_1,\ldots,\tau'_r\}$ of $\mathfrak{T}_E$ with
$(\sum_{i=1}^r Q_{\tau'_i})(m) = m$ we have
$\{\tau_1,\ldots,\tau_n\}\subseteq \{\tau'_1,\ldots,\tau'_r\}$.}
\end{remark}

\begin{definition} For every $T \in \mathcal{E}$ we define the map $E_{\rm ap}:\mathcal{E} \rightarrow \mathcal{E}$, $T \mapsto E_{\rm ap}(T)$ as follows: for any $m \in M$,
$$E_{\rm ap}(T)(m):= (\sum_{i=1}^n Q_{\tau_i}TQ_{\tau_i})(m).$$
\end{definition}

\begin{proposition}\label{Eap} Let $\mathcal{Q} = \{Q_\tau\}_{\tau \in \mathfrak{T}_E} \subset \mathcal{E}$.
Then the map $E_{\rm ap}$ is an algebraic conditional expectation of $\mathcal{E}$ onto $\mathcal{Q'}$ where
$$\mathcal{Q'}=\{T \in \mathcal{E}: TQ_{\tau} = Q_{\tau}T \text{ for every } \tau \in \mathfrak{T}_E\}.$$
\end{proposition}

\begin{proof} We need to check the conditions from the definition of algebraic conditional expectation above.
\begin{itemize}
\item[(i)] $(E_{\rm ap})^2 = E_{\rm ap}$ follows from the fact that the elements in the collection $(Q_\tau)_{\tau \in \mathfrak{T}_E}$ are mutually orthogonal idempotents. Let us prove that  ${\rm Im}(E_{\rm ap}) = \mathcal{Q'}$:

    Consider $T \in {\rm Im}(E_{\rm ap})$, that is, $T = E_{\rm ap}(T')$ for some $T' \in \mathcal{E}$. For $m \in M$ there exist $\{\tau_1,\ldots,\tau_n\} \subset \mathfrak{T}_E$ such that $(\sum_{i=1}^n Q_{\tau_i})(m) = m$, so we have $T(m)= E_{\rm ap}(T')(m) = (\sum_{i=1}^n Q_{\tau_i}T'Q_{\tau_i})(m)$. Let us check that $T \in \mathcal{Q'}$: consider $Q_\tau$ for some $\tau \in \mathfrak{T}_E$ then
    $$(Q_\tau T)(m)= Q_\tau((\sum_{i=1}^n Q_{\tau_i}T'Q_{\tau_i})(m))= (\sum_{i=1}^n Q_\tau Q_{\tau_i}T'Q_{\tau_i})(m).$$
    Since the elements in the collection $(Q_\tau)_{\tau \in \mathfrak{T}_E}$ are mutually orthogonal idempotents,
    $$T ((Q_\tau)(m))= ( Q_{\tau}T'Q_{\tau})(Q_\tau(m))= Q_{\tau}T'Q_{\tau}(m).$$
    Also since $(\sum_{i=1}^n Q_{\tau_i})(m) = m$, then $Q_{\tau}(m)=0$ if $\tau\neq \tau_i$ for each $i$.
    In any case since the collection $(Q_\tau)_{\tau \in \mathfrak{T}_E}$ consists of mutually orthogonal idempotents,
    \begin{align*}
    (T Q_\tau)(m) = (Q_\tau T)(m)  & = \begin{cases}
    Q_\tau T' Q_\tau (m) & \text{if }\ \tau = \tau_i \text{ for some } i=1,\ldots,n \\
    0 & \text{otherwise. }\  \end{cases}
   \end{align*}
    Conversely consider $T \in \mathcal{Q'}$. For $m \in M$ then again there exist $\{\tau_1,\ldots,\tau_n\} \subset \mathfrak{T}_E$ such that $(\sum_{i=1}^n Q_{\tau_i})(m) = m$. So
    \begin{align*}
    & E_{\rm ap}(T)(m)= (\sum_{i=1}^n Q_{\tau_i}TQ_{\tau_i})(m) = (\sum_{i=1}^n T Q_{\tau_i}Q_{\tau_i})(m) =\\
    & = (\sum_{i=1}^n T Q_{\tau_i})(m) = T (\sum_{i=1}^n Q_{\tau_i})(m)= T(m).
    \end{align*}
    Which means that $T \in {\rm Im}(E_{\rm ap})$.
\item[(ii)] Let $T' \in \mathcal{Q'}$. For $T \in \mathcal{E}$ and for $m \in M$ there exist $\{\tau_1,\ldots,\tau_n\} \subset \mathfrak{T}_E$ such that $(\sum_{i=1}^n Q_{\tau_i})(m) = m$. Then we have
\begin{align*}
& E_{\rm ap}(T'T)(m) = (\sum_{i=1}^n Q_{\tau_i}T'TQ_{\tau_i})(m) = (\sum_{i=1}^n T' Q_{\tau_i}TQ_{\tau_i})(m) =\\
& = T'(\sum_{i=1}^n Q_{\tau_i}TQ_{\tau_i})(m) = T' E_{\rm ap}(T)(m).
\end{align*}
Similarly it can be shown that $E_{\rm ap}(TT')(m)= E_{\rm
ap}(T)T'(m)$. \qedhere
\end{itemize}
\end{proof}

 As in Proposition \ref{Piaprepresentation}, let $M$ be the $R$-module generated by the set $\{\xi_{\tau}^n: n \in \mathbb{Z}, \tau \in \mathfrak{T}_E\}$. Given a path $\alpha \in \text{Path}(E)$, we define $P_{\alpha} \in \mathcal{E}$, $P_{\alpha}: M \rightarrow M$ in the following way: for every $\tau \in \mathfrak{T}_E$ and every $n \in \mathbb{Z}$,
\begin{align*}
P_\alpha(\xi_{\tau}^n)  & = \begin{cases}
\xi_{\tau}^n & \text{if }\ \alpha \leq \tau, \\
0 & \text{otherwise. }\  \end{cases}
\end{align*}

Take into account that we have $P_\alpha = S_\alpha S_\alpha^{\ast}$.

We are now in the position to show the existence of an algebraic conditional expectation onto the subalgebra $M_R(E)$.

\begin{theorem}\label{EM} There exists a unique algebraic conditional expectation $E_M$ of $L_R(E)$ onto $M_R(E)$ such that
\begin{align*}
E_M(x)  & = \begin{cases}
x & \text{if }\ x \in G_E^M,\\
0 & \text{if }\ x \in G_E \setminus G_E^M. \end{cases}
\end{align*}
Furthermore, for the essentially aperiodic representation $\Pi_{\rm ap}: L_R(E) \rightarrow \mathcal{E}$ and the algebraic conditional expectation $E_{\rm ap}: \mathcal{E} \rightarrow \mathcal{E}$ we have that
\begin{itemize}
\item[(i)] $\Pi_{\rm ap} \circ E_M = E_{\rm ap} \circ \Pi_{\rm ap}$, and
\item[(ii)] $E_{\rm ap}(T) = T$ for every $T \in \mathcal{P'}$ where
$$\mathcal{P'}=\{T \in \mathcal{E}: TP_{\alpha} = P_{\alpha}T \text{ for every } \alpha \in \text{Path}(E)\}.$$
\end{itemize}
\end{theorem}

\begin{proof} First let us prove the following claims:
\begin{itemize}
\item[(a)] $E_{\rm ap}(\Pi_{\rm ap}(x)) = \Pi_{\rm ap}(x)$ for every $x \in M_R(E)$.\\
 Consider the collection $\mathcal{P} = \{P_{\alpha}\}_{\alpha \in \text{Path}(E)} = \{\Pi_{\rm ap}(\alpha \alpha^{\ast})\}_{\alpha \in \text{Path}(E)}$. Notice that
 \begin{align*}
Q_\tau\Pi_{\rm ap}(\alpha \alpha^{\ast})= \Pi_{\rm ap}(\alpha \alpha^{\ast})Q_\tau & = \begin{cases}
Q_\tau & \text{if }\ \alpha \leq \tau,\\
0 & \text{otherwise. }\  \end{cases}
\end{align*}
 Thus $\mathcal{P} \subset \mathcal{Q'}$, where
$\mathcal{Q'}=\{T \in \mathcal{E}: TQ_{\tau} = Q_{\tau}T \text{ for every } \tau \in \mathfrak{T}_E\}$.

Suppose $T \in \mathcal{E}$ such that $T$ commutes with all $P_{\alpha}$ , $\alpha \in \text{Path}(E)$. In particular for every $\tau \in \mathfrak{T}_E$ we have that $TP_{\tau_{(n)}}=P_{\tau_{(n)}}T$ for every $n\geq 0$. Observe that for any $m \in M$ there exists some integer $N \geq 0$ such that for every $n \geq N$, $P_{\tau_{(n)}}(m)=Q_{\tau}(m)$. So we get $TQ_\tau = Q_\tau T$ which implies $\mathcal{P'}\subset \mathcal{Q'}$ where
$$\mathcal{P'}=\{T \in \mathcal{E}: TP_{\alpha} = P_{\alpha}T \text{ for every } \alpha \in \text{Path}(E)\}.$$
Therefore we have that
$$E_{\rm ap}(T) = T \text{ for every } T \in \mathcal{P'}.$$

In particular for all $x \in M_R(E)$, it follows from Proposition \ref{CommutativeSubalgebra} that $\Pi_{\rm ap}(x) \in \mathcal{P'}$ so we get
$$E_{\rm ap}(\Pi_{\rm ap}(x)) = \Pi_{\rm ap}(x), \text{ for every } x \in M_R(E).$$

\item[(b)] $E_{\rm ap}(\Pi_{\rm ap}(x)) = 0$ for every  $x \in G_E \setminus G_E^M$.\\
Consider $x = \alpha \beta^{\ast} \in G_E\setminus G_E^M$, that is, $r(\alpha) = r(\beta)$. Let $\tau \in \mathfrak{T}_E$ then clearly $Q_\tau S_\alpha S_\beta^{\ast} Q_\tau \neq 0$ only if $\beta \leq \tau$ and $\alpha \leq \tau$. In any case then we will have that $\alpha \leq \beta$ or $\beta \leq \alpha$ and since $x\not\in G_E^M$, $\alpha\neq\beta$. We can suppose the case $\beta \leq \alpha$; then it follows that $s(\alpha) = s(\beta)$ and necessarily $\alpha = \beta \lambda$ for some closed path $\lambda$. Since $\beta \leq \tau$ assume that $\tau = \beta \mu$ for some path $\mu$. The fact that $Q_\tau S_\alpha S_\beta^{\ast} Q_\tau \neq 0$ implies that $\beta \lambda \mu = \alpha \mu = \tau=\beta\mu$, and this yields $\lambda \mu = \mu$ which means that $\tau$ is periodic with period $\lambda$. Since $\tau$ is essentially aperiodic then $\lambda$ has no exits. Then $\mu=r(\lambda)$ and so $\lambda\in E^0$. Thus $\alpha=\beta$ which is impossible. Finally we have that then $E_{\rm ap}(\Pi_{\rm ap}(x)) = 0$ for every $x \in G_E \setminus G_E^M$.
\end{itemize}

So we have that $E_{\rm ap}(\Pi_{\rm ap}(G_E)) \subseteq \Pi_{\rm ap}(G_E)$. Then extending linearly we get that $E_{\rm ap}(\Pi_{\rm ap}(L_R(E))) \subseteq \Pi_{\rm ap}(L_R(E))$. Since $\Pi_{\rm ap}$ is injective by Proposition \ref{Piaprepresentation}, then there exists a unique linear map $E_M : L_R(E) \rightarrow L_R(E)$ such that $\Pi_{\rm ap} \circ E_M = E_{\rm ap} \circ \Pi_{\rm ap}$. Besides, by the injectivity of $\Pi_{\rm ap}$ and by (a) and (b) we have
\begin{align*}
E_M(x)  & = \begin{cases}
x & \text{if }\ x \in G_E^M,\\
0 & \text{if }\ x \in G_E \setminus G_E^M. \end{cases}
\end{align*}

It then is immediate that $E_M$ is an algebraic conditional expectation of $L_R(E)$ onto $M_R(E)$.
\end{proof}

We can prove now the main result of this section.

\begin{theorem}\label{Mismaximal} Let $E$ be a graph and $R$ a commutative ring with unit. Consider $M_R(E) \subseteq L_R(E)$. Then
$$M_R(E) = \{x \in L_R(E): xd=dx \text{ for every } d \in \Delta(E)\}.$$
Furthermore, $M_R(E)$ is a maximal commutative subalgebra of $L_R(E)$.
\end{theorem}

\begin{proof} For the equality $M_R(E) = \{x \in L_R(E): xd=dx \text{ for every } d \in \Delta(E)\}$ we need to prove the double inclusion. First it is clear that
$$M_R(E) \subseteq \{x \in L_R(E): xd=dx \text{ for every } d \in \Delta(E)\}$$
 since $\Delta(E) \subset M_R(E)$ and $M_R(E)$ is commutative by Proposition \ref{CommutativeSubalgebra}.

For the other containment consider an element $x \in L_R(E)$ such that $xd=dx$ for every $d \in \Delta(E)$. Then for every path $\alpha \in \text{Path}(E)$ we know that $\alpha \alpha^{\ast} \in \Delta(E)$ and then
$$\Pi_{\rm ap}(x) P_\alpha =\Pi_{\rm ap}(x)\Pi_{\rm ap}(\alpha \alpha^{\ast})= \Pi_{\rm ap}(x \alpha \alpha^{\ast}) = \Pi_{\rm ap}(\alpha \alpha^{\ast} x) = P_\alpha \Pi_{\rm ap}(x).$$
 By Theorem \ref{EM}, we know that $E_{\rm ap}(T) = T \text{ for every } T \in \mathcal{P'}$ where $$\mathcal{P'}=\{T \in \mathcal{E}: TP_{\alpha} = P_{\alpha}T \text{ for every } \alpha \in \text{Path}(E)\}.$$
In particular it follows that $\Pi_{\rm ap}(x) = E_{\rm ap}(\Pi_{\rm ap}(x))$, using Theorem \ref{EM}, $E_{\rm ap}(\Pi_{\rm ap}(x)) = \Pi_{\rm ap}(E_M(x))$ and hence $\Pi_{\rm ap}(x) = \Pi_{\rm ap}(E_M(x))$. Since $\Pi_{\rm ap}$ is injective by Proposition \ref{Piaprepresentation}, we have that $x = E_M(x) \in M_R(E)$ and we are done.

In order to prove that $M_R(E)$ is a maximal commutative subalgebra inside $L_R(E)$, consider $\mathcal{C}$ a commutative subalgebra of $L_R(E)$ such that $M_R(E) \subseteq \mathcal{C}$. Since we have $\Delta(E) \subseteq M_R(E) \subseteq \mathcal{C}$ then in particular
$$\mathcal{C} \subseteq \{x \in L_R(E): xd=dx \text{ for every } d \in \Delta(E)\} = M_R(E).$$
So in the end necessarily $\mathcal{C} = M_R(E)$.
\end{proof}

The following example shows that the core is not the unique maximal commutative subalgebra of $L_R(E)$.
\begin{example}
\rm Let $E$ be the following graph:
$$
\xymatrix{
 \bullet_{a}\ar[r]^{\alpha} &\bullet_{b} \\
}
$$

and $R$ be a commutative ring. Then $L_R(E)\cong \mathbb{M}_2(R)$ and $M_R(E)$ is isomorphic to the subalgebra $B=\{(b_{ij})| (b_{ij})\in \mathbb{M}_2(R), b_{12}=b_{21}=0\}$ of $\mathbb{M}_2(R)$. Let $A=\{(a_{ij})| (a_{ij})\in \mathbb{M}_2(R), a_{11}=a_{22}, a_{21}=0\}$ be the subalgebra of $\mathbb{M}_2(R)$. Then $A$ is a commutative subalgebra of $\mathbb{M}_2(R)$ that $A$ is not a subalgebra of $B$. Therefore $M_R(E)$ in general is not a unique maximal commutative subalgebra of $L_R(E)$.

\end{example}

We state the following consequence. The center of a Leavitt path algebra was studied in \cite{AC} and \cite{CMMSS}.

\begin{corollary} The center of the Leavitt path algebra $L_R(E)$, that is
 $$Z(L_R(E))= \{ x \in L_R(E): xy = yx \text{ for every } y \in L_R(E)\}$$
 satisfies that $Z(L_R(E)) \subseteq M_R(E)$.
\end{corollary}

It is an open question as to which $R$ and $E$ have the property that $Z(L_R(E)) = M_R(E)$.

\begin{remark}{\rm Using the new approach of the commutative core of $L_R(E)$, we can re-establish the structure of commutative Leavitt path algebras (originally given in \cite[Proposition 2.7]{AC}) and commutative Cohn path algebras.}
\end{remark}

To finish we include a specific example in order to illustrate some of the ideas.
\begin{example} \rm Consider the following graph $E$:
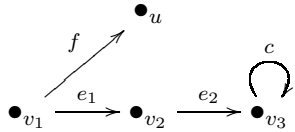
$$
\xymatrix{
 & \bullet_{u} & \\
 \bullet_{v_1} \ar[ur]^{f} \ar[r]^{e_1} &  \bullet_{v_2} \ar[r]^{e_2}  & \bullet_{v_3} \ar@(ul,ur)^{c}
}
$$

First observe that the set of trails in $E$ are given by $\{u, f, ccc\cdots, e_2ccc\cdots, e_1e_2ccc\cdots\}$. By Proposition \ref{PropMaxTotOrdSubset} we can determine the maximal totally ordered subsets of the diagonal generator set $G_E^{\Delta}$:

$\mathcal{P}_{u}=\{u\}$,

$\mathcal{P}_{f}=\{v_1,ff^{\ast}\}$,

$\mathcal{P}_{ccc\cdots}=\{v_3, cc^{\ast}, cc(cc)^{\ast}, \ldots\}=\{v_3\}$ (since $cc^{\ast}=v_3$),

$\mathcal{P}_{e_2ccc\cdots}=\{v_2, e_2e_2^{\ast}, e_2c(e_2c)^{\ast}, e_2cc(e_2cc)^{\ast}, \ldots\}=\{v_2\}$ (since $cc^{\ast}=v_3$ and $e_2e_2^{\ast}=v_2$) and

$\mathcal{P}_{e_1e_2ccc\cdots}=\{v_1, e_1e_1^{\ast}, e_1e_2(e_1e_2)^{\ast}, e_1e_2c(e_1e_2c)^{\ast}, e_1e_2cc(e_1e_2cc)^{\ast}, \ldots\}=\{v_1, e_1e_1^{\ast}\}$ (since again $e_2e_2^{\ast}=v_2$ and $cc^{\ast}=v_3$).

According to Definition \ref{EssentiallyAperiodicTrail} we identify the periodic trails as the set

$\{ccc\cdots,e_2ccc\cdots,e_1e_2ccc\cdots\}$. (Notice that the seeds of $ccc\cdots$, $e_2ccc\cdots$, and $e_1e_2ccc..$ are respectively $(v_3,c)$, $(e_2,c)$, and $(e_1e_2,c)$). And the set of essentially aperiodic trails is given by the following:
$$\mathfrak{T}_E=\{u,f,ccc\cdots,e_2ccc\cdots,e_1e_2ccc\cdots\}.$$

In this case there are not continuous essentially aperiodic trails. By Definition \ref{Defdistinguished}, the distinguished paths in $E$ are then $\{v_3,e_2,e_1e_2\}$.

Remember $M_R(E)$ is the $R$-algebra generated by $G_E^M$, $M_R(E) = <G_E^M> $. We compute the normal generators $G_E^M$. We know
$G_E^M = G_E^\Delta \cup \{ \omega_\alpha^n : \alpha \text{ is a distinguished path, } n \neq 0\}$. Therefore:
 $$G_E^M = \{u, v_1, v_2, v_3, ff^{\ast}, e_1e_1^{\ast}\} \cup \{c^n, e_2c^ne_2^\ast, e_1e_2c^n(e_1e_2)^{\ast}\} \text{ with } 0 \neq n \in \mathbb{Z}.$$

To finish with this example let us check that the core $M_R(E)$ and the center $Z(L_R(E))$ are different subalgebras of $L_R(E)$, that is, $Z(L_R(E)) \subsetneq M_R(E)$: consider $e_1e_1^{\ast} \in M_R(E)$ but $e_1e_1^{\ast}e_1 = e_1 \neq 0 = e_1e_1e_1^{\ast}$ so then $e_1e_1^{\ast} \notin Z(L_R(E))$.

\end{example}

\section{General Uniqueness Theorem for Leavitt path algebras}

In this final section we focus our attention on giving a new version of the Uniqueness Theorem for Leavitt path algebras; concretely we will prove that we only need to check the injectivity of a homomorphism when we restrict to the commutative core $M_R(E)$. For this purpose first we will write here the so-called Reduction Theorem for Leavitt path algebras but referred to the case over a commutative ring with unit; the proof is followed similarly to that given in \cite[Theorem 2.2.11]{AAS} so we will omit it here.

\begin{theorem}\label{ReductionThm} Let $E$ be an arbitrary graph and $R$ a commutative ring with unit. For any non-zero element $a \in L_R(E)$ there exist $\mu, \nu \in \text{Path}(E)$ such that either:
\begin{itemize}
\item[(i)]$0 \neq \mu^{\ast}a\nu = rv$, for some $r \in R \setminus \{0\}$ and $v \in E^0$, or
\item[(ii)]$0 \neq \mu^{\ast
}a\nu = p(\lambda)$, where $\lambda$ is a cycle without exits and $p(x)$ is a non-zero polynomial in $R[x,x^{-1}]$.
\end{itemize}
\end{theorem}

\begin{theorem}\label{UniquenessThm} Let $E$ be a graph and $R$ a commutative ring with unit. Consider $\Phi: L_R(E) \rightarrow \mathcal{A}$ a ring homomorphism. Then the following conditions are equivalent:
\begin{itemize}
\item[(i)] $\Phi$ is injective;
\item[(ii)] the restriction of $\Phi$ to $M_R(E)$ is injective;
\item[(iii)] both these conditions are satisfied:
\begin{itemize}
\item[(a)] $\Phi(rv) \neq 0$, for all $v \in E^0$ and for all $r \in R \setminus \{0\}$;
\item[(b)] for every distinguished path $\alpha$ the $\ast$ $R$-algebra $< \Phi(\omega_\alpha) >$ generated by $\Phi(\omega_\alpha)$ is $\ast$-isomorphic to $R[x,x^{-1}]$; that is, $< \Phi(\omega_\alpha) > \ \cong R[x,x^{-1}]$.
\end{itemize}
\end{itemize}
\end{theorem}

\begin{proof}Let us prove the equivalence between the statements.

(i) $\Rightarrow$ (ii)  It is obvious.

(ii) $\Rightarrow$ (iii) First let us check (a). Consider $rv$, for some $v \in E^0$ and $r \in R \setminus \{0\}$. But $rv = rvv^{\ast} \in \Delta(E) \subset M_R(E)$, so it follows immediately.

Now in order to see (b), note that for every distinguished path $\alpha$ and each $i\in \mathbb{Z}$ then $\Phi(\omega^i_\alpha) \neq 0$: this is straightforward since every $\omega_\alpha \in M_R(E)$ by Proposition \ref{CommutativeSubalgebra}.

Then we have a natural $\ast$-isomorphism $< \omega_\alpha > \ \cong \ < \Phi(\omega_\alpha) >$. We know on the other hand that $< \omega_\alpha > \ \cong R[x,x^{-1}]$ is a $\ast$-isomorphism by Remark \ref{Omega_alpha}. It follows therefore that $< \Phi(\omega_\alpha) > \ \cong R[x,x^{-1}]$ as desired.

(iii) $\Rightarrow$ (i) Suppose $0 \neq a \in L_R(E)$ such that $a \in {\rm Ker}\ \Phi$. By Theorem \ref{ReductionThm} there exist $\mu, \nu \in \text{Path}(E)$ and we have two possibilities.

 Assume $0 \neq \mu^{\ast}a\nu = rv$, for some $r \in R \setminus \{0\}$ and $v \in E^0$. But $a \in {\rm Ker}\ \Phi$, so $rv \in {\rm Ker}\ \Phi$ which is a contradiction with hypothesis (a).

 So necessarily $0 \neq \mu^{\ast}a\nu = p(\lambda)$, where $\lambda$ is a cycle without exits and $p(x)$ is a non-zero polynomial in $R[x,x^{-1}]$. Since $a \in {\rm Ker}\ \Phi$ we have $p(\lambda) \in {\rm Ker}\ \Phi$, so $\Phi(p(\lambda))=0$. Let $v\in \lambda^0$, then $v$ is a distinguished vertex and $\lambda=\lambda_v$, so by hypothesis (b) it is satisfied that $ g:<\Phi(\lambda)> \rightarrow R[x,x^{-1}]$ is an algebra isomorphism. We know $0 = \Phi(p(\lambda))$, then $g(\Phi(p(\lambda)))=0$. Assume that $p(\lambda)=r_{-m}\lambda^{-m}+ \ldots +r_n\lambda^{n}$ for certain  $m,n \in \mathbb{N}$. Then
 $0=g(\Phi(p(\lambda)))=g(\Phi(r_{-m} v)\Phi(\lambda^*)^{m}+....+\Phi(r_n v)\Phi(\lambda)^{n})=g(\Phi(r_{-m} v))x^{-m}+....+g(\Phi(r_n v))x^{n})$. This means that $g(\Phi(r_{i} v))= 0$ for all $i \in \{-m, \ldots, n\}$ which implies that $\Phi(r_{i} v)= 0$ for each $i$ giving us a contradiction with (a). \qedhere
\end{proof}

As a consequence we can obtain the Cuntz-Krieger Uniqueness Theorem (see \cite[Theorem 6.5]{T}). If the graph $E$ satisfies Condition (L) then there are no distinguished paths in $\text{Path}(E)$, so we get immediately:

\begin{theorem}{\rm (Cuntz-Krieger Uniqueness Theorem)} Let $E$ be a graph satisfying Condition (L) and let $R$ be a commutative ring with unit. Suppose $\Phi: L_R(E) \rightarrow \mathcal{A}$ is a ring homomorphism. Then the following conditions are equivalent:
\begin{itemize}
\item[(i)] $\Phi$ is injective;
\item[(ii)] the restriction of $\Phi$ to $M_R(E)$ is injective;
\item[(iii)] $\Phi(rv) \neq 0$, for all $v \in E^0$ and for all $r \in R \setminus \{0\}$.
\end{itemize}
\end{theorem}

By using the proof of (ii) $\Rightarrow$ (iii) in the proof of Theorem \ref{UniquenessThm} we can reformulate the Graded Uniqueness Theorem:

\begin{theorem}{\rm (Graded Uniqueness Theorem)}\label{GradedUniquenessThm} Let $E$ be a graph and $R$ a commutative ring with unit. If $\mathcal{A}$ is a $\mathbb{Z}$-graded ring and $\Phi: L_R(E) \rightarrow \mathcal{A}$ is a $\mathbb{Z}$-graded ring homomorphism, then the following conditions are equivalent:
\begin{itemize}
\item[(i)] $\Phi$ is injective;
\item[(ii)] the restriction of $\Phi$ to $M_R(E)$ is injective;
\item[(iii)] $\Phi(rv) \neq 0$, for all $v \in E^0$ and for all $r \in R \setminus \{0\}$.
\end{itemize}
\end{theorem}


\section*{acknowledgements}

The authors would like to thank to the referee for a careful reading of this paper and making many helpful suggestions which led to improve and reorganize it. Also we are thankful to Gonzalo Aranda Pino for his comments on a earlier version of this paper. The first author was partially supported by the Spanish MEC and Fondos FEDER through project MTM2013-41208-P, and by the Junta de Andaluc\'{\i}a and Fondos FEDER, jointly, through project FQM-7156. Part of this work was carried out during a visit of the first author to the Institute for Research in Fundamental Sciences (IPM-Isfahan) in Isfahan, Iran. The first author thanks this host institution for its warm hospitality and support. The research of the second author was in part supported by a grant from IPM (No. 94170419).

\end{document}